\documentclass[12pt]{amsart}

\usepackage{amssymb}
\usepackage{mathtools}
\usepackage{hyperref}

\usepackage{enumitem}

\usepackage{graphicx}

\makeatletter
\@namedef{subjclassname@2020}{%
  \textup{2020} Mathematics Subject Classification}
\makeatother

\usepackage[T1]{fontenc}

\let\strokel\l
\renewcommand{\l}{{\ell}}
\renewcommand{\L}{{\lambda}}
\renewcommand{\Re}{{\operatorname{Re}}}
\renewcommand{\Im}{{\operatorname{Im}}}
\renewcommand{\a}{\alpha}

\DeclarePairedDelimiter\floor{\lfloor}{\rfloor}

\newtheorem{thm}{Theorem}[section]

\newtheorem{lem}[thm]{Lemma}

\newtheorem{conjec}[thm]{Conjecture}

\theoremstyle{definition}

\newtheorem{remark}[thm]{Remark}

\numberwithin{equation}{section}

\begin{document}


\baselineskip=17pt


\title[First moment]{The first moment of quadratic twists of modular $L$-functions}

\author[Q. Shen]{Quanli Shen}
\address{SDU-ANU Joint Science College\\ Shandong University\\
 Weihai 264209, China}
\email{qlshen@outlook.com}

\date{}

\begin{abstract}
We obtain an  asymptotic formula with an error term $O(X^{\frac{1}{2} + \varepsilon})$ for the smoothed first moment of quadratic twists of modular $L$-functions.  We also give a similar result for the smoothed first moment of the first derivative of quadratic twists of modular $L$-functions. The argument is largely based on  Young's works (Acta Arith 138(1):73--99, 2009 and  Selecta Math 19(2):509--543, 2013).
\end{abstract}



\maketitle

\section{Introduction}
The study of moments of $L$-functions is of much  interest to  researchers in number theory due to its fruitful applications.
One example is that Bump-Friedberg-Hoffstein \cite{Bump-F-H} and Murty-Murty \cite{Murty-Murty} independently proved $L'(\frac{1}{2},E \otimes \chi_d) \neq 0$ for infinitely many  fundamental discriminants $d$ with $d<0$, where $E$ is a modular elliptic curve with root number $1$ over $\mathbb{Q}$ and $\chi_d (\cdot) := \left(\frac{d}{\cdot} \right)$ denotes the Kronecker symbol. The method of their works is to investigate moments of the  derivative of quadratic twists of modular $L$-functions. Their celebrated results successfully verify the assumption in Kolyvagin's 
theorem \cite{Kolyvagin}  on the Birch-Swinnerton-Dyer conjecture, where it was proven that if the  Hasse-Weil $L$-function $L(s,E)$ does not vanish at the central point $s =\frac{1}{2}$, then the group of rational points of $E$ is finite, \textit{provided} that there exists a quadratic character $\chi_d$ with $d<0$ such that $L(s,E \otimes \chi_d)$ has a simple zero at the central point and such that $\chi_d(p) = 1 $ for every $p $ that divides the conductor of $E$.

In particular,   Murty-Murty \cite{Murty-Murty} proved an asymptotic formula for the first moment of the derivative of quadratic twists of modular $L$-functions with an error term $O(X (\log X)^{1-\rho})$, where $\rho$ is an explicit positive real number. It was later  improved by Iwaniec \cite{Iwaniec} to a power savings $O(X^{\frac{13}{14} + \varepsilon})$ for a smoothed version.  In \cite{Bump-F-H-01} Bump-Friedberg-Hoffstein
claimed the error term $O(X^{\frac{3}{5} + \varepsilon})$ without proof.
Note that  in \cite{Iwaniec,Murty-Murty} they  considered quadratic twists of elliptic curve $L$-functions,  but it is no doubt that the methods there will extend to all modular newforms. The goal of this paper is to obtain an error term of the size $O(X^{\frac{1}{2}+\varepsilon})$ for a smoothed version. The improvement is mainly due to a recursive method developed by Heath-Brown  \cite{HB01} and Young \cite{Young01, Young}. The argument of this paper also allows us to obtain an error term of the same size $O(X^{\frac{1}{2} + \varepsilon})$ for the first moment of quadratic twists of modular $L$-functions, which improves the  error term  $O(X^{\frac{13}{14} + \varepsilon})$ of Stefanicki \cite[Theorem 3]{Stefanicki} and Luo-Ramakrishnan \cite[Proposition 3.6]{Luo-Ramakrishnan} and $O(X^{\frac{7}{8} + \varepsilon})$ of {Radziwi\strokel\strokel}-Soundararajan \cite[Proposition 2]{Sound-Radziwill}. Also, with slightly more effort, one can obtain  similar results for the first moment of higher derivatives of twisted modular $L$-functions.

To precisely state our result, we shall introduce some notation. Let $f$ be a modular form of weight $\kappa $ for the full modular group $SL_2 (\mathbb{Z})$. (Our argument  may extend to congruent subgroups.) We assume $f$ is an eigenfunction of all Hecke operators. The Fourier expansion of $f$ at infinity is 
\[
f(z) = \sum_{n=1}^{\infty} \L_f (n) n^{\frac{\kappa -1}{2}} e(nz),
\]
where $\L_f (1) =1$ and $|\L_f(n)| \leq \tau(n)$ for $n \geq 1$. Here $e(z) := e^{2 \pi i z}$, and $\tau(n)$ is the number of divisors of $n$. The twisted  modular $L$-function is defined  by 
\begin{align*}
L(s, f \otimes \chi_d) &:= \sum_{n=1}^{\infty} \frac{\L_f(n)\chi_d(n)}{n^s}
 = \prod_{p\nmid d} \left(1 - \frac{\L_f (p) \chi_d(p)}{p^s}  + \frac{1}{p^{2s}}\right)^{-1}
\end{align*} 
for $\Re(s) > 1$, 
and it extends to the entire complex plane. The completed $L$-function is defined by 
\begin{align*}
\Lambda (s, f \otimes \chi_d) := \left(\frac{|d|}{2\pi} \right)^s \Gamma (s + \tfrac{\kappa -1}{2}) L(s, f \otimes \chi_d). 
\end{align*}
It  satisfies the functional equation 
\begin{align}
\Lambda (s, f \otimes \chi_d) = i^\kappa \epsilon(d ) \Lambda (1- s, f \otimes \chi_d),
\label{equ:FE}
\end{align}
where $\epsilon(d) =1$ if $d$ is positive, and  $\epsilon(d) =-1$ if $d$ is negative. In this paper, we  consider the case $d>0$, so $\epsilon =1$. The case $d<0$ can be done similarly.
We prove the following assertions.
\begin{thm}\label{main-thm}
Let  $\kappa \equiv 0\, (\operatorname{mod}4)$ and $\kappa \neq 0$. Let $\Phi(x): (0,\infty) \rightarrow \mathbb{R}$ be a smooth, compactly supported function. We have 
\begin{align*}
\sideset{}{^*}\sum_{(d,2)=1} L(\tfrac{1}{2}, f \otimes \chi_{8d}) \Phi(\tfrac{d}{X})=  \frac{8 \tilde{\Phi}(1)}{\pi^2} L(1, \operatorname{sym}^2 f)  Z^* (0)  X + O(X^{\frac{1}{2} + \varepsilon}).
\end{align*}
Here  $\sideset{}{^*}\sum$ denotes the summation over square-free integers, $Z^*$ is defined in  \eqref{equ:LZ} and \eqref{def-Zstar}, and $\tilde{\Phi}$ is the Mellin transform of $\Phi$ defined by 
\[
\tilde{\Phi} (s) := \int_0^\infty \Phi(x) x^{s-1} dx.
\]

\end{thm}
\begin{thm}\label{main-thm1}
Let  $\kappa \equiv 2\, (\operatorname{mod}4)$. Let $\Phi(x): (0,\infty) \rightarrow \mathbb{R}$ be a smooth, compactly supported function. We have 
\begin{align*}
&\sideset{}{^*}\sum_{(d,2)=1} L'(\tfrac{1}{2}, f \otimes \chi_{8d}) \Phi(\tfrac{d}{X})\\
&=  \frac{8 \tilde{\Phi}(1)}{\pi^2} L(1, \operatorname{sym}^2 f)  Z ^ *(0)  X  
\Bigg[ 
\log X
 + 2 \frac{L'(1, \operatorname{sym}^2 f) }{L(1, \operatorname{sym}^2 f) }
+ \frac{{Z^*}' (0)}{Z^* (0)}\\
&  + \log \frac{8}{2\pi}
+ \frac{\Gamma'(\frac{\kappa}{2})}{\Gamma(\frac{\kappa}{2})}
+ \frac{\tilde{\Phi}'(1)}{\tilde{\Phi}(1)}
\Bigg]+ O(X^{\frac{1}{2} + \varepsilon}).
\end{align*}
\end{thm}
In the above, the symmetric square $L$-function is defined by 
\begin{align*}
L(s, \operatorname{sym}^2 f) &:= \zeta(2s) \sum_{n=1}^{\infty} \frac{\L_f( n^2)}{n^{s}}\\& = \prod_p \left( 1 - \frac{\alpha_f(p)^2}{p^s}\right)^{-1} \left( 1 - \frac{\alpha_f(p) \beta_f(p)}{p^s}\right)^{-1} \left( 1 - \frac{\beta_f(p)^2}{p^s}\right)^{-1},
\end{align*}
where $\Re(s)>1$, $\alpha_f(p ) + \beta_f(p ) = \L_f(p) $ and $\alpha_f(p )  \beta_f(p ) =1$. We see that the main term in Theorem \ref{main-thm1} coincides with \cite[Theorem 2.3]{Petrow}. Note that  in \cite[Theorem 2.3]{Petrow} the form of the moment and the definition of $Z^*$  is slightly different from ours.

It is worth mentioning that recently Bui--Florea--Keating--Roditty-Gershon \cite{BFKR}  obtained the error term of the same size $O(X^{\frac{1}{2}+\varepsilon})$ for the function field analogue.  The second moment, expected to be much more difficult, was computed asymptotically by Soundararajan and Young \cite{Sound-Young} under the generalized Riemann hypothesis. Their method was also used by Petrow \cite{Petrow} for studying moments of derivatives of twisted modular $L$-functions. The computation of asymptotic formulas for higher moments is believed beyond current techniques,  whereas we do have beautiful conjectures due to Keating-Snaith  \cite{Keating-Snaith02}  and Conrey-Farmer-Keating-Rubinstein-Snaith  \cite{Conrey-Farmer-Keating-Rubinstein-Snaith}.

The moments of quadratic twists of modular $L$-functions  are  comparable to the  moments of quadratic Dirichlet $L$-functions.  An iterative method, pioneered  by Heath-Brown \cite{HB01} to study mean values of real characters,  was further developed by Young \cite{Young01} to obtain an  error term $O(X^{\frac{1}{2} + \varepsilon})$ in  an asymptotic formula for the first moment of quadratic Dirichlet $L$-functions.  The error term  $O(X^{\frac{1}{2} + \varepsilon})$  was also essentially implicit in Goldfeld-Hoffstein's work \cite{Goldfeld-Hoffstein}. In addition, by using the recursive method, the third moment  of quadratic Dirichlet $L$-functions was improved to $O(X^{\frac{3}{4}+ \varepsilon})$ by Young \cite{Young},  and recently the second moment was improved to $O(X^{\frac{1}{2}+ \varepsilon})$ by Sono \cite{Sono}.  The moment  in Theorem \ref{main-thm} is analogous to the second moment of quadratic Dirichlet $L$-functions, so it should not be a coincidence that Sono's work \cite{Sono} and Theorem \ref{main-thm}  have the same error term $O(X^{\frac{1}{2}+ \varepsilon})$. The conjectured error term for the second moment of quadratic Dirichlet $L$-functions is $O(X^{\frac{1}{2}+ \varepsilon})$ (see Alderson-Rubinstein \cite{Alderson-Rubinstein}),  so it may be hard to improve  Theorems \ref{main-thm} and  \ref{main-thm1}. On a more fundamental level, it is because we do not know how to obtain an error term better than $O(X^{\frac{1}{2}+o(1)})$ unconditionally in the problem of counting square-free integers with a smooth weight.

The proof for Theorems \ref{main-thm} and \ref{main-thm1} is similar to \cite{Sono, Young01, Young}.     To adapt to the recursive method, we consider the shifted first moment twisted by a quadratic character as follows:
\begin{align}
M(\alpha, \ell): = \sideset{}{^*}\sum_{(d,2)=1} \chi_{8d} (\ell)L(\tfrac{1}{2} + \alpha,f \otimes \chi_{8d})\Phi(\tfrac{d}{X}),
\label{equ:M-l}
\end{align}
where $\l$ is a positive, odd integer. Write $\l = \l_1 \l_2^2$, where $\l_1$ is square-free. We may make the following conjecture.
\begin{conjec}
\label{conj}
Let $h \geq \frac{1}{2}$. Let $\Phi(x): (0,\infty) \rightarrow \mathbb{C}$ be a smooth, compactly supported function. 
Assume $|\Re(\a)| \ll \frac{1}{\log X}$ and $|\Im(\a)| \ll (\log X)^2$. Then for any $\varepsilon>0$, we have
\begin{equation}
\begin{split}
&M(\alpha, \ell) \\
 &= \frac{4X \tilde{\Phi}(1)}{\pi^2 \l_1 ^{\frac{1}{2} + \a}} L(1+2\alpha, \operatorname{sym}^2 f)  Z (\tfrac{1}{2} + \alpha, \l) \\
 &  +  i^\kappa  \frac{4 \gamma_\a X^{1-2\a} \tilde{\Phi}(1-2\a)}{\pi^2 \l_1 ^{\frac{1}{2} - \a}} L(1-2\alpha, \operatorname{sym}^2 f)  Z (\tfrac{1}{2} - \alpha, \l) + O(\l^{\frac{1}{2} + \varepsilon}X^{h+\varepsilon}).
\label{assump}
\end{split}
\end{equation}
Here the big $O$ is depending on $\varepsilon, h $ and $\Phi$. The symbol $\gamma_\a$ is defined in \eqref{gamma-alpha}. For $\Re(\gamma)>0$, 
\begin{align}
L(1+2\gamma, \operatorname{sym}^2 f)  Z (\tfrac{1}{2} + \gamma,\l) := \prod_{(p,2)=1}  Z_{p} (\tfrac{1}{2} +\gamma,  \l),
\label{equ:LZ}
\end{align}
where for $p| \l_1$,
\begin{equation}
\begin{split}
&Z_{p} (\tfrac{1}{2} + \gamma,  \l):=\\
&p^{\frac{1}{2} + \gamma} \left( \frac{p}{p+1} \right) \left[ \frac{1}{2} \left( 1- \frac{\L_f (p)}{p^{\frac{1}{2} + \gamma} } + \frac{1}{p^{1+2\gamma}}\right) ^{-1}- \frac{1}{2 } \left( 1+ \frac{\L_f (p)}{p^{\frac{1}{2} + \gamma} }+ \frac{1}{p^{1+2\gamma}}\right ) ^{-1} \right ],
\end{split}
 \label{equ:conj1}
\end{equation} 
for $ p\nmid \l_1, p| \l_2$,
\begin{equation}
\begin{split}
&Z_{p} (\tfrac{1}{2} + \gamma,  \l)\\
&:=
  \frac{p}{p+1}\left[ \frac{1}{2} \left( 1- \frac{\L_f (p)}{p^{\frac{1}{2} + \gamma} } + \frac{1}{p^{1+2\gamma}}\right) ^{-1}+ \frac{1}{2 } \left( 1+ \frac{\L_f (p)}{p^{\frac{1}{2} + \gamma} }+ \frac{1}{p^{1+2\gamma}}\right) ^{-1} \right] ,
\end{split}
\label{equ:conj2}
\end{equation}
and for $(p, 2\l) =1$,
\begin{equation}
\begin{split}
&Z_{p} (\tfrac{1}{2} + \gamma,  \l):=\\
&
 1 +  \frac{p}{p+1}\left[ \frac{1}{2} \left( 1- \frac{\L_f (p)}{p^{\frac{1}{2} + \gamma} } + \frac{1}{p^{1+2\gamma}}\right) ^{-1}+ \frac{1}{2 } \left( 1+ \frac{\L_f (p)}{p^{\frac{1}{2} + \gamma} }+ \frac{1}{p^{1+2\gamma}}\right) ^{-1} -1\right].
\end{split}
\label{equ:conj3}
\end{equation}
%

The function $Z (\tfrac{1}{2} + \gamma, \l)$ is analytic and absolutely convergent in the region $\Re(\gamma) > -\frac{1}{4}$.
\end{conjec}
The main term in \eqref{assump} can be conjectured by heuristically following this paper's argument  or using the recipe method in \cite{Conrey-Farmer-Keating-Rubinstein-Snaith}.
To obtain Theorems \ref{main-thm} and \ref{main-thm1}, it suffices to prove the following theorem.
\begin{thm}\label{thm-assump}
If Conjecture \ref{conj} is true for some $h \geq \frac{1}{2}$, then it is true for $\frac{4h-1}{4h}$ replacing $h$.
\end{thm}
\begin{proof}[Proof of Theorems  \ref{main-thm} and \ref{main-thm1}]
 We see  Conjecture \ref{conj} is true  for $h=1$ by Lemma \ref{HB-bd}  in the next section. By Theorem \ref{thm-assump} we can reduce it to $h= 1, \frac{3}{4}, \frac{2}{3}, \cdots$, which tends to $h = \frac{1}{2}$.   Set $\l=1$ and write 
 \begin{align}
 Z^*( \alpha) := Z(\tfrac{1}{2} + \a, 1). 
 \label{def-Zstar}
 \end{align}
 Then Theorem \ref{main-thm} follows by letting $\a \rightarrow 0$ in \eqref{assump}.  We can differentiate both sides of \eqref{assump} in terms of $\a$. Note that the error term in \eqref{assump} is holomorphic on the disc centred at $(0,0)$ with radius $\ll \frac{1}{\log X}$. Hence the size of the derivative of the error term is still $O(X^{\frac{1}{2} + \varepsilon})$ by  Cauchy's integral formula. This gives Theorem \ref{main-thm1} by letting $\a \rightarrow 0$. Note that  we can compute asymptotic formulas for the first moment of higher derivatives of twisted modular $L$-functions in a similar way.
\end{proof}
The rest of the paper will focus on proving Theorem \ref{thm-assump}. The idea is as follows. We first apply the approximate functional equation in the twisted $L$-function in \eqref{equ:M-l}. Then the M\"{o}bius inversion is used  to remove the square-free condition where the new parameter $a$ is introduced.  We split the summation over $a$ into two pieces. For large $a$, the Poisson summation formula is employed to separate the summation into diagonal terms and non-diagonal terms (see their definitions below \eqref{equ:M-1plus}).  On the other hand, for small $a$, we convert the summation  back to that with the square-free condition, where we will use the induction hypothesis \eqref{assump}. We obtain partial main terms and error terms there. These partial  main terms  can be perfectly combined with the diagonal terms  after some simplification, finally leading to the main term in  \eqref{assump}.  We remark that  there is  perfect  cancellation between various subsidiary terms  in the moments of quadratic Dirichlet $L$-functions (see \cite{Young01,Young, Sono}).  Note that it is hard to estimate these terms individually with an error better than $O(X^{\frac{3}{4} + \varepsilon})$. In our paper, we do not find this cancellation.  Fortunately, the subsidiary terms in this paper can be bounded by $O(X^{\frac{1}{2}+ \varepsilon})$, which is  small enough for our purpose.

\section{Preliminary lemmas}
\begin{lem}\label{lem:AFE}
Let $G(s)$ be an even, entire function with $G(0)=1$, bounded in any fixed strip $| \Re(s)| \leq A$,  and decaying rapidly as $|\Im (s)| \rightarrow \infty$. Let 
\begin{align*}
\omega_\alpha (\xi) &:= \frac{1}{2\pi i} \int_{(1)} \frac{G(s)}{s} g_\alpha (s) \xi^{-s} ds,
\end{align*}
where
\begin{align*}
g_\alpha(s) &:= (2\pi)^{-s} \frac{\Gamma (\frac{\kappa}{2} + \alpha + s)}{\Gamma(\frac{\kappa}{2} + \alpha)},
\end{align*}
and 
$\int_{(c)}$  denotes the contour integral $\int_{c - i \infty} ^{c+ i \infty}$. Set 
\begin{align*}
X_{\alpha,d} &:= \left( \frac{|d|}{2\pi}\right)^{-2\alpha} \frac{\Gamma(\frac{\kappa}{2} - \alpha)}{\Gamma (\frac{\kappa}{2} + \alpha)}.
\end{align*}
 Then we  have 
\begin{align*}
&L(\tfrac{1}{2} + \alpha, f \otimes \chi_d)\\
& =\sum_{n=1}^{\infty} \frac{\lambda_f (n)\chi_d (n)}{n^{\frac{1}{2} + \alpha}} \omega_\alpha \left( \frac{n}{|d|} \right) +  i^\kappa \epsilon(d) X_{\alpha,d} \sum_{n=1}^{\infty} \frac{\lambda_f(n)\chi_d(n)}{n^{\frac{1}{2} - \alpha}} \omega_{-\alpha} \left( \frac{n}{|d|}\right).
\end{align*}
\end{lem}
\begin{proof}
Set 
\begin{equation}
\begin{split}
I &:= \frac{1}{2\pi i} \int_{(1)} \left( \frac{|d|}{2\pi} \right)^s \frac{\Gamma ( \frac{\kappa}{2} + \alpha + s)}{\Gamma(\frac{\kappa}{2} + \alpha)} L (\tfrac{1}{2} + \alpha + s, f \otimes \chi_d) \frac{G(s)}{s} ds \\
 &=\frac{1}{2 \pi i} \int_{(1)}\frac{\Lambda (\frac{1}{2} + \alpha +s)}{\Gamma(\frac{\kappa}{2} + \alpha)} \frac{G(s)}{s} \left( \frac{d}{2\pi} \right) ^{-\frac{1}{2} - \alpha} ds.
 \end{split}
\label{equ:I}
\end{equation}
Move the line of integration to $\Re(s) = -1$. The residue theorem gives
\begin{align}
L(\tfrac{1}{2} + \alpha , f \otimes \chi_d)=  I- I',
\label{equ:I-I'}
\end{align}
where 
\begin{align*}
I' := \frac{1}{2\pi i} \int_{(-1)}  \frac{\Lambda (\frac{1}{2} + \alpha +s)}{\Gamma(\frac{\kappa}{2} + \alpha)} \frac{G(s)}{s} \left( \frac{d}{2\pi} \right) ^{-\frac{1}{2} - \alpha} ds.
\end{align*}
By changing the variable  $s \rightarrow -s$ and the functional equation \eqref{equ:FE}, 
\begin{equation}
\begin{split}
I' &= - i^\kappa \epsilon(d)  \frac{1}{2\pi i} \int_{(1)}  \frac{\Lambda (\frac{1}{2} - \alpha +s)}{\Gamma(\frac{\kappa}{2} + \alpha)} \frac{G(s)}{s} \left( \frac{|d|}{2\pi} \right) ^{-\frac{1}{2} - \alpha} ds\\
&= -  i^\kappa  \epsilon(d) X_{\alpha,d} \frac{1}{2\pi i} \int_{(1)} \left( \frac{|d|}{2\pi} \right)^s \frac{\Gamma ( \frac{\kappa}{2} - \alpha + s)}{\Gamma(\frac{\kappa}{2} - \alpha)} L (\tfrac{1}{2} - \alpha + s, f \otimes \chi_d) \frac{G(s)}{s} ds.
\end{split}
\label{equ:I'}
\end{equation}
Insert \eqref{equ:I} and \eqref{equ:I'} back  into  \eqref{equ:I-I'}  and write  $L (\tfrac{1}{2} \pm \alpha + s, f \otimes \chi_d)$ as  their Dirichlet series. This completes the proof.
\end{proof}
\begin{remark}\label{rem:G}
Write 
\begin{align*}
\mathcal{Z}(\a, s) :=\zeta(2 + 4 \alpha + 4s)(1+ 4 \alpha + 4s)  (1-4\a -4s).
\end{align*}
We can take 
\begin{equation*}
\begin{split}
G(s) &= e^{s^2}  \frac{\mathcal{Z}(\a, s) \mathcal{Z}(\a, - s) \cdot \mathcal{Z}(- \a, s)\mathcal{Z}(-\a, -s)}{\mathcal{Z}(\a, 0)^2 \mathcal{Z}(-\a, 0)^2}.
\end{split}
\end{equation*}
The purpose of adding some zeta factors into $G(s)$ is that they cancel out certain terms in $Z(\frac{1}{2} \pm \alpha \pm s)$.  See Lemma \ref{lem:lastlem} and \eqref{equ:above-final2} for an example.
\end{remark}
The following lemma is a generalized version of Poisson summation formula established by Soundararajan \cite[Lemma 2.6]{Sound00} (also see \cite[Lemma 2.2]{Sound-Young}).
\begin{lem}
 Let $\Phi$ be a smooth function with compact support on the positive real numbers, and suppose that $n$
 is an odd integer. Then
 \[
  \sum_{(d,2)=1}\left(\frac{d}{n}\right)\Phi\left(\frac{d}{Z}\right)=\frac{Z}{2n}\left(\frac{2}          {n}\right) \sum_{k\in \mathbb{Z}}(-1)^kG_{k}(n)\hat{\Phi}\left(\frac{kZ}{2n}\right),
 \]
where
 \begin{align*}
  G_k(n):=\left(\frac{1-i}{2}+\left(\frac{-1}{n}\right)\frac{1+i}{2}\right)\sum_{a\, (\operatorname{mod}{n})}
  \left(\frac{a}{n}\right)e\left(\frac{ak}{n}\right),
  \label{equ:def-G}
 \end{align*}
and
 \[
  \hat{\Phi}(y):=\int_{-\infty}^{\infty} \left(\cos(2\pi xy)+\sin(2\pi xy) \right)\Phi(x)dx
 \]
is a Fourier-type transform of $\Phi$. 
\label{lem:Poisson}
\end{lem}
The Gauss-type sum $G_k(n)$ above can be explicitly computed in the following lemma (see \cite[Lemma 2.3]{Sound00}).
\begin{lem}
 If $m$ and $n$ are relatively prime odd integers, then $G_k(mn)=G_k(m)G_k(n)$. Moreover, if $p^{\alpha}$ is 
 the largest power of $p$ dividing $k$ (setting $\alpha=\infty$ if $k=0$), then 
 \begin{align*}
 G_k(p^{\beta})=\left\{
 \begin{array}
  [c]{ll}
  0 & \text{if }\beta\leq\alpha \text{ is odd},\\
  \phi(p^{\beta}) & \text{if }\beta\leq\alpha \text{ is even},\\
  -p^{\alpha} & \text{if }\beta=\alpha+1 \text{ is even},\\
  \left(\frac{kp^{-\alpha}}{p}\right)p^{\alpha}\sqrt{p} & \text{if }\beta=\alpha+1 \text{ is odd},\\
  0 & \text{if }\beta\geq\alpha+2.
 \end{array}
 \right.
\end{align*}
\label{lem:preciseG}
Here $\phi$ is the Euler totient function.
\end{lem}
We need the following upper bound for the first moment of twisted modular $L$-functions. It is analogous to  \cite[Theorem 2]{HB01} of Heath-Brown. 
\begin{lem}\label{HB-bd}
For $\sigma\geq \frac{1}{2}$, we have 
\[
\sideset{}{^\flat} \sum_{|d| \leq X} |L(\sigma + it, f \otimes \chi_d)| \ll_\varepsilon X^{1 + \varepsilon} (1+ |t|)^{\frac{1}{2} + \varepsilon},
\]
where $\sideset{}{^\flat}\sum$ denotes the summation over fundamental discriminants.
\end{lem}
\begin{proof}
It  follows from  \cite[Corollary 2.5]{Sound-Young} and the Cauchy-Schwarz inequality.
\end{proof}

\begin{lem}\label{lem:lastlem}
Let $\Re(\gamma) > 0$. Then for any integer $N\geq 0$, we have
\begin{align*}
&L(1+2\gamma, \operatorname{sym}^2 f)  Z (\tfrac{1}{2} + \gamma,\l) \\
&  = L(1+2\gamma, \operatorname{sym}^2 f)  \frac{\zeta(2^{N+1} + 2^{N+2}\gamma) }{\zeta(2+4\gamma)} Z^N(\tfrac{1}{2} + \gamma,\l).
\end{align*}
Here $Z^N(\tfrac{1}{2} + \gamma,\l)$ is analytic  and  is bounded by $\l^ \varepsilon $ in the region 
$\Re(\gamma) > \max ( - \frac{1}{2} + \varepsilon, - \frac{1}{2} + \frac{1}{2^{N+2}} + \frac{\varepsilon}{2^{N+2}} )$. Note $Z^0 = Z$.
\end{lem}
\begin{proof}
Recall the definition of $L(1+2\gamma, \operatorname{sym}^2 f)  Z (\tfrac{1}{2} + \gamma,\l) $ in \eqref{equ:LZ} and the expression of $Z_{p} (\tfrac{1}{2} + \gamma,\l)$ in \eqref{equ:conj1}--\eqref{equ:conj3}. 
It is clear that in the region $\Re(\gamma) > 0 $,  for $(p,2\l)=1$, 
\begin{equation}
\begin{split}
Z_{p} (\tfrac{1}{2} + \gamma,\l)
&= \left(1 - \frac{\L_f(p)}{p^{\frac{1}{2} + \gamma}}  + \frac{1}{p^{1+2\gamma}}\right)^{-1}
\left(1 + \frac{\L_f(p)}{p^{\frac{1}{2} + \gamma}}  + \frac{1}{p^{1+2\gamma}}\right)^{-1}
\\&\times \left( 1+ \frac{1}{p^{1 + 2 \gamma}}+ P(p,\gamma) \right),
\end{split}
\label{equ:Pgamma}
\end{equation}
where 
\begin{align*}
&P(p,\gamma)  \\
&= -\frac{1}{p+1} \left[ 1 + \frac{1}{p^{1+2\gamma}} - \left(1 - \frac{\L_f(p)}{p^{\frac{1}{2} + \gamma}}  + \frac{1}{p^{1+2\gamma}}\right)
\left(1 + \frac{\L_f(p)}{p^{\frac{1}{2} + \gamma}}  + \frac{1}{p^{1+2\gamma}}\right)
\right].
\end{align*}
We see that $P(p,\gamma)  = O (\frac{1}{p^{1+ \varepsilon}})$ when $\Re(\gamma) > -\frac{1}{2} + \varepsilon$.
We then factor out $1+ \frac{1}{p^{1 + 2 \gamma}}$ from the last brackets of \eqref{equ:Pgamma}. It gives
\begin{equation}
\begin{split}
1+ \frac{1}{p^{1 + 2 \gamma}}+ P(p,\gamma)
&= 
\frac{1}{ 1- \frac{1}{p^{1 + 2 \gamma}} }\left( 1- \frac{1}{p^{1 + 2 \gamma}}  \right)\left( 1+ \frac{1}{p^{1 + 2 \gamma}}+ P(p,\gamma) \right) \\
&= \frac{1}{ 1- \frac{1}{p^{1 +2 \gamma}} }  \left( 1- \frac{1}{p^{2 + 2^{2} \gamma}}   + P(p,\gamma) - \frac{1}{p^{1+ 2\gamma}} P(p,\gamma)\right).
\end{split}
\label{chapter1:Pgamma}
\end{equation}
Note  that $P(p,\gamma) - \frac{1}{p^{1+ 2\gamma}} P(p,\gamma)  = O (\frac{1}{p^{1+ \varepsilon}})$ when $\Re(\gamma) > -\frac{1}{2} + \varepsilon$.
One can check  that  the expressions $1- \frac{1}{p^{2 + 2^{2} \gamma}}   + P(p,\gamma) - \frac{1}{p^{1+ 2\gamma}} P(p,\gamma) $ are exactly the Euler factors of $Z (\tfrac{1}{2} + \gamma,\l)$ in the cases of $(p,2\l)=1$.

Factoring  out $1- \frac{1}{p^{2 + 2^{2} \gamma}} $ from \eqref{chapter1:Pgamma}, we obtain
\begin{align*}
&1- \frac{1}{p^{2 + 2^{2} \gamma}}   + P(p,\gamma) - \frac{1}{p^{1+ 2\gamma}} P(p,\gamma) \\
&= \frac{1}{1+ \frac{1}{p^{2 + 2^{2} \gamma}}} \left( 1+ \frac{1}{p^{2 + 2^{2} \gamma}}\right)
   \left(  1- \frac{1}{p^{2 + 2^{2} \gamma}}   + P(p,\gamma) - \frac{1}{p^{1+ 2\gamma}} P(p,\gamma) \right) \\
&=      \frac{1}{1+ \frac{1}{p^{2 + 2^{2} \gamma}}} 
          \left( 1- \frac{1}{p^{2^2 + 2^{3} \gamma}}+ P(p,\gamma) - \frac{1}{p^{1+2\gamma}} P(p,\gamma) + \frac{1}{p^{2+2\gamma}}P(p,\gamma)\right. \\
         &  \left. - \frac{1}{p^{3+6\gamma}} P(p,\gamma) \right).
\end{align*}
Note  that those terms involving $P(p,\gamma)$ above  are $O (\frac{1}{p^{1+ \varepsilon}})$ when $\Re(\gamma) > -\frac{1}{2} + \varepsilon$. 
Repeating the above process continuously we can get for $N\geq 1$,
\begin{align*}
& \frac{1}{1+ \frac{1}{p^{2 + 2^{2} \gamma}}} \\
&\times
          \left( 1- \frac{1}{p^{2^2 + 2^{3} \gamma}}+ P(p,\gamma) - \frac{1}{p^{1+2\gamma}} P(p,\gamma) + \frac{1}{p^{2+2\gamma}}P(p,\gamma) - \frac{1}{p^{3+6\gamma}} P(p,\gamma) \right)\\
&= 
\prod_{m=1}^N \frac{1}{1+ \frac{1}{p^{2^m + 2^{m+1} \gamma}}}\left( 1 - \frac{1}{p^{2^{N+1} + 2^{N+2} \gamma}}+ Q(p,\gamma) \right)\\
&= \frac{1- \frac{1}{p^{2+ 4\gamma}}}{1- \frac{1}{p^{2^{N+1} + 2^{N+2} \gamma}}}\left(  1- \frac{1}{p^{2^{N+1} + 2^{N+2} \gamma}}+ Q(p,\gamma) \right),
\end{align*}
where $Q(p,\gamma)$ can be written explicitly which satisfies $Q(p,\gamma) =O_{N,\varepsilon} (\frac{1}{p^{1+ \varepsilon}})$ when $\Re(\gamma) > -\frac{1}{2} + \varepsilon$. The last equation is obtained via multiplying by $ \frac{1- \frac{1}{p^{2+2^2 \gamma}}}{1- \frac{1}{p^{2 + 2^2 \gamma}}}$.

Note the expressions $1- \frac{1}{p^{2^{N+1} + 2^{N+2} \gamma}}+ Q(p,\gamma) $ are exactly the Euler factors of $Z^N(\tfrac{1}{2} + \gamma,\l)$ when $(p,2\l)=1$. For $\Re(\gamma) > \max ( - \frac{1}{2} + \varepsilon, - \frac{1}{2} + \frac{1}{2^{N+2}} + \frac{\varepsilon}{2^{N+2}} )$, we have 
\[
\prod_{(p,2\l)=1} \left( 1- \frac{1}{p^{2^{N+1} + 2^{N+2} \gamma}}+Q(p,\gamma) \right) \ll 1.
\]

In addition, it is easy to derive the Euler factors of $Z^N(\tfrac{1}{2} + \gamma,\l)$ corresponding to  $p|2\l$ and  prove that  they contribute  $ \ll \l^{\varepsilon}$, as desired.
\end{proof}

\section{Setup of the problem}
By \eqref{equ:M-l}  and Lemma \ref{lem:AFE},  we get
\begin{align*}
M(\a,\l) = M^+ (\a,\l) + M^-(\a,\l),
\end{align*}
where 
\begin{align*}
M^+ (\a,\l) &:= \sideset{}{^*}\sum_{(d,2)=1} \Phi\left(\frac{d}{X} \right)   \sum_{n=1}^\infty \frac{\L_f(n) \chi_{8d} ( \l n)}{n^{\frac{1}{2} + \a}} \omega_\a \left( \frac{n}{8d}\right),\\
M^- (\a,\l) &:= i^\kappa  \sideset{}{^*}\sum_{(d,2)=1} \Phi\left(\frac{d}{X} \right)  X_{\a,8d} \sum_{n=1}^\infty \frac{\L_f(n) \chi_{8d} ( \l n )}{n^{\frac{1}{2} - \a}} \omega_{-\a} \left( \frac{n}{8d}\right).
\end{align*}
\begin{remark} \label{rem:Phi}
Define $\Phi_z (x) := x^z \Phi(x)$. We then can write
\[
M^- (\a,\l) = i^\kappa \gamma_\alpha X^{-2\alpha} \sideset{}{^*}\sum_{(d,2)=1} \Phi_{-2\alpha} \left(\frac{d}{X} \right)   \sum_{n=1}^\infty \frac{\L_f(n) \chi_{8d} ( \l n)}{n^{\frac{1}{2} - \a}} \omega_{-\a} \left( \frac{n}{8d}\right),
\]
where 
\begin{align}
\gamma_\alpha := \left( \frac{8}{2\pi}\right)^{-2\alpha} \frac{\Gamma(\frac{\kappa}{2} - \alpha)}{\Gamma (\frac{\kappa}{2} + \alpha)} .
\label{gamma-alpha}
\end{align}
Notice that  $M^- (\a,\l)$ is equal to  $ i^\kappa \gamma_\a X^{-2\a} M^+ (-\a,\l)$ with $\Phi_{-2\a}$  in place of $\Phi$. Thus we just need to evaluate  $M^+ (\a,\l)$, and the results for $M^- (\a,\l)$ can be obtained immediately.
\end{remark}

The square-free condition in $M^+ (\a,\l) $ can be removed by using  M\"{o}bius inversion. This gives
\begin{equation}
\begin{split}
M^+ (\a,\l) &= \sum_{(d,2) =1} \sum_{a^2 |d} \mu(a)\Phi\left(\frac{d}{X} \right)   \sum_{n=1}^\infty \frac{\L_f(n) \chi_{8d} ( \l n )}{n^{\frac{1}{2} + \a}} \omega_\a \left( \frac{n}{8d}\right) \\
&= \sum_{(a,2\l) =1} \mu (a) \sum_{(d,2) =1} \sum_{(n,a)=1} \frac{\L_f (n) \chi_{8d} (\l n )}{n^{\frac{1}{2} + \a}} \omega_\a \left( \frac{n}{8a^2 d}\right) \Phi\left(\frac{a^2d}{X} \right)\\
&=: M_N^+ (\a,\l) + M_R^+ (\a,\l),
\end{split}
\label{equ:M1-M2}
\end{equation}
where $M_N^+ (\a,\l)$ and $ M_R^+ (\a,\l)$ denote the sums over $a \leq Y$ and $a>Y$, respectively. Here $Y (\leq X)$ is a parameter chosen later.

We  use the  Poisson summation formula to split $M_N^+ (\a,\l)$. Using  Lemma \ref{lem:Poisson} on the summation over $d$ in  $M_N^+ (\a,\l)$,  we derive
\begin{equation}
\begin{split}
M_N^+ (\a,\l)& = \frac{X}{2} \sum_{\substack{(a,2\l) =1\\ a \leq Y}} \frac{\mu(a)}{a^2} \sum_{(n,2a)=1} \frac{\L_f(n)}{n^{\frac{1}{2} + \a}} \sum_{k \in \mathbb{Z}} (-1)^k  \frac{G_k(\l n )}{\l n} \\
&\times \int_{-\infty}^{\infty} (\cos + \sin ) \left( \frac{2 \pi k x X}{2 n \l a^2}\right) \omega_\a \left( \frac{n}{8xX}\right) \Phi(x) dx.
\end{split}
\label{equ:M-1plus}
\end{equation}
Let $M_N^+ (\a,\l, k=0)$ denote the term $k=0$ above, and let $M_N^+ (\a,\l, k \neq 0)$ denote the remaining terms. We call $M_N^+ (\a,\l, k=0)$ diagonal terms and $M_N^+ (\a,\l, k \neq 0)$ off-diagonal terms.

On the other hand, we convert $M_R^+ (\a,\l)$ in \eqref{equ:M1-M2} back to the summation over square-free integers, and then appeal to  the induction hypothesis \eqref{assump}. To see this, recall that 
\begin{align*}
M_R^+ (\a,\l) = \sum_{\substack{(a,2\l) =1 \\ a >Y}} \mu (a) \sum_{(d,2) =1} \sum_{(n,a)=1} \frac{\L_f (n) \chi_{8d} (\l n )}{n^{\frac{1}{2} + \a}} \omega_\a \left( \frac{n}{8a^2 d}\right) \Phi\left(\frac{a^2d}{X} \right).
\end{align*}
Write $d = eb^2$, where $e$ is square-free and $b$ is positive. Group terms according to $c=ab$. It follows that
\begin{equation}
\begin{split}
M_R^+ (\a,\l) &= \sum_{(c,2\l)=1} \sum_{\substack{a>Y \\ a|c}} \mu (a)  \sideset{}{^*}\sum_{(e,2)=1} \sum_{(n,2c)=1} \frac{\L_f (n) \chi_{8e} (\l n)}{n^{\frac{1}{2} + \a}} \omega_\a \left( \frac{n}{8c^2e }\right) \Phi \left( \frac{c^2 e}{X}\right) \\
&= \sum_{(c,2\l)=1} \sum_{\substack{a>Y \\ a|c}} \mu (a)  \frac{1}{2 \pi i}  \int_{(1)}\sideset{}{^*}  \sum_{(e,2)=1} \chi_{8e} (\l)  \Phi_s \left( \frac{ e}{X'}\right)\\
& \times  L_c (\tfrac{1}{2} + \alpha +s, f \otimes \chi_{8e}) X^s 8^sg_\a (s)\frac{G(s)}{s} ds,
\end{split}
\label{equ:M2-1}
\end{equation}
where 
$X' : = \frac{X}{c^2}$.
Here $L_c (s, f \otimes \chi_{8e}),\  \Re(s)>1$ is given by the Euler product of $L (s, f \otimes \chi_{8e})$ with omitting all
prime factors of $c$. In the first equation above, the condition $(c,\l)=1$ is due to $\chi_{8ed^2} (\l) =0$ when $(d,\l) \neq 1$. We use the following lemma  to change $L_c (\tfrac{1}{2} + \alpha +s, f \otimes \chi_{8e}) $ back to the form of $L (\tfrac{1}{2} + \alpha +s, f \otimes \chi_{8e}) $. It is similar to  \cite[Lemma 9]{Kowalski-Michel} of Kowalski and Michel.
\begin{lem}\label{lem:local}
Let $d$ be a fundamental discriminant. Then
\begin{equation}
\begin{split}
&\prod_{p|c} \left( 1 - \frac{\L_f(p) \chi_d (p)}{p^s} + \frac{\chi_d (p^2)}{p^{2s}}\right) \\
&= \sum_{m|c}\sum_{n|c} \mu(m)\mu (mn)^2 \L_f (m) \chi_d(m) \chi_d (n^2)\frac{1}{m^s} \frac{1}{n^{2s}}.
\end{split}
\label{equ:local}
\end{equation}
\end{lem}
\begin{proof}
Note that the summand on the right-hand side of \eqref{equ:local} is jointly multiplicative. Thus
\begin{align*}
&\sum_{m|c}\sum_{n|c} \mu(m)\mu (mn)^2 \L_f (m) \chi_d(m) \chi_d (n^2)\frac{1}{m^s} \frac{1}{n^{2s}} \\
&= \prod_{p|c} \sum_{0 \leq r_1,r_2 \leq \operatorname{order}_p(c)} \mu(p^{r_1})\mu (p^{r_1 + r_2})^2 \L_f (p^{r_1}) \chi_d(p^{r_1}) \chi_d (p^{2r_2})\frac{1}{p^{r_1s }} \frac{1}{p^{2r_2s}} \\
&= \prod_{p|c} \left( 1 - \frac{\L_f(p) \chi_d (p)}{p^s} + \frac{\chi_d (p^2)}{p^{2s}}\right),
\end{align*}
as desired.
\end{proof}
It follows from \eqref{equ:M2-1}  and Lemma \ref{lem:local}  that
\begin{align*}
M_R^+ (\a,\l) 
&= \sum_{(c,2\l)=1} \sum_{\substack{a>Y \\ a|c}} \mu (a)  \sum_{r_1 |c}  \frac{\mu(r_1)  \L_f (r_1) }{r_1^{\frac{1}{2} + \alpha }}   
\sum_{r_2 |c}   \frac{\mu( r_1 r_2)^2}{r_2^{1+ 2\a }}  \\
& \times \frac{1}{2 \pi i}  \int_{(\frac{1}{\log X})}
   \sideset{}{^*}\sum_{(e,2)=1} \chi_{8e} (\l_1r_1 \l_2^2r_2^2)  \Phi_s \left( \frac{ e}{X'}\right) L (\tfrac{1}{2} + \alpha +s, f \otimes \chi_{8e})\\
   & \times \frac{1}{r_1^sr_2^{2s}} X^s 8^sg_\a (s)\frac{G(s)}{s} ds. 
\end{align*}
We can truncate the above integral for $|\Im(s)| \ll (\log X)^2$ with an error $O(1)$ by the rapid decay of $|G(s)|$ as $|\Im(s)| \rightarrow \infty$.    For  $|\Im(s)| \ll (\log X)^2$, we are allowed to employ the inductive hypothesis \eqref{assump}. Hence  we have
\begin{align*}
M_R^+ (\a,\l) = M_{R,1}^+ (\a,\l)  + M_{R,2}^+ (\a,\l) + M_{R,3}^+ (\a,\l) +O(1),
\end{align*}
where 
\begin{align}
M_{R,1}^+ (\a,\l)
&:=\frac{1}{\l_1^{\frac{1}{2} + \alpha}} \sum_{(c,2\l)=1} \sum_{\substack{a>Y \\ a|c}} \mu (a)  \sum_{r_1 |c}  \frac{\mu(r_1)  \L_f (r_1) }{r_1^{1+ 2\alpha }}   
\sum_{r_2 |c}   \frac{\mu( r_1 r_2)^2}{r_2^{1+ 2\a }} \frac{1}{2 \pi i}  \int_{(\frac{1}{\log X})}  \nonumber \\
& \times      \frac{4X^{1+s} \tilde{\Phi}(1+s)}{\pi^2c^2} L(1+2\alpha+2s, \operatorname{sym}^2 f)  Z (\tfrac{1}{2} + \alpha+s,\l r_1 r_2^2 )\nonumber\\
& \times
\frac{8^s}{\l_1^sr_1^{2s}r_2^{2s}} g_\a (s)\frac{G(s)}{s} ds, \label{equ:M21}\\
M_{R,2}^+ (\a,\l) &:=\frac{ i^\kappa}{\l_1^{\frac{1}{2} - \alpha }} \sum_{(c,2\l)=1} \sum_{\substack{a>Y \\ a|c}} \mu (a)  \sum_{r_1 |c}  \frac{\mu(r_1)  \L_f (r_1) }{r_1}   
\sum_{r_2 |c}   \frac{\mu( r_1 r_2)^2}{r_2^{1+ 2\a}}\nonumber \\
& \times \frac{1}{2 \pi i}  \int_{(\frac{1}{\log X})}\frac{4X ^{1-2\a-s}\gamma_{\a +s}}{\pi^2c^{2-4\a -4s}} 
      \tilde{\Phi}(1-2\a-s)   \nonumber \\
      & \times L(1-2\alpha-2s, \operatorname{sym}^2 f)Z (\tfrac{1}{2} - \alpha-s,\l r_1r_2^2)
\frac{\l_1^s 8^s }{r_2^{2s}}  g_\a (s)\frac{G(s)}{s} ds, \label{equ:M22}\\
M_{R,3}^+ (\a,\l) 
&:= \sum_{(c,2\l)=1} \sum_{\substack{a>Y \\ a|c}} \mu (a)  \sum_{r_1 |c}  \frac{\mu(r_1)  \L_f (r_1) }{r_1^{\frac{1}{2} + \alpha }}   
\sum_{r_2 |c}   \frac{\mu( r_1 r_2)^2}{r_2^{1+ 2\a }} \frac{1}{2 \pi i}  \int_{\substack{s= \frac{1}{\log X} + it \\ |t| \ll (\log X)^2 }} \nonumber\\
& \times   
O\left( ( \l r_1r_2^2)^{\frac{1}{2} + \frac{\varepsilon}{100}} X'^{h+ \varepsilon}\right)
\frac{1}{r_1^sr_2^{2s}} X^s 8^sg_\a (s)\frac{G(s)}{s} ds.  \label{equ:M23}
\end{align}
Note that in \eqref{equ:M21} and \eqref{equ:M22} we  have extended the range of  integrals from $|\Im(s)| \ll (\log X)^2$ to  the vertical line $\Re(s) = \frac{1}{\log X}$  with an error $O(1)$.

Now we have separated $M^+ (\a,\l)$ into several parts. In summary, we have obtained
\begin{align}
M(\a,\l) = M^+ (\a,\l) + M^-(\a,\l),
\label{relation1}
\end{align} 
and
\begin{equation}
\begin{split}
M^{+} (\a,\l) &= M_N^{+} (\a,\l) + M_R^{+} (\a,\l),\\
M_N^{+} (\a,\l) &=  M_N^{+} (\a,\l, k=0) +  M_N^{+} (\a,\l, k\neq 0),\\
M_R^{+} (\a,\l) &= M_{R,1}^{+} (\a,\l)  + M_{R,2}^{+} (\a,\l) + M_{R,3}^{+} (\a,\l) +O(1).
\end{split}
\label{relation2}
\end{equation}
We can also split $M^-(\a,\l)$ similarly using Remark \ref{rem:Phi}. We will evaluate $M_N^{+} (\a,\l, k=0)$, $M_N^{+} (\a,\l, k\neq 0)$, respectively,  in Sections \ref{sec:k=0} and  \ref{sec:knot=0}. 
The analysis for $M_{R,1}^{+} (\a,\l)$,  $M_{R,2}^{+} (\a,\l)$ and  $M_{R,3}^{+} (\a,\l)$ will be done in Section \ref{sec:Error}.  We complete the proof of Theorem \ref{thm-assump}  in Section \ref{sec:comp}.

\section{Evaluation of $M_N^+ (\a,\l, k=0)$  \label{sec:k=0}}
Recall $M_N^+ (\a,\l, k=0)$ in \eqref{equ:M-1plus}.  By the definition of  $G_k(n)$ in Lemma \ref{lem:Poisson}, we know $G_0(n) =\phi (n)$ if $n=\Box$, and $G_0(n) =0$ otherwise. Here $n=\Box$ means $n$ is a square number.   Hence
\begin{equation}
\begin{split}
&M_N^+ (\a,\l, k=0) \\
&= \frac{X}{2} \sum_{\substack{(a,2\l)  =1 \\ a \leq Y}}  \frac{\mu(a)}{a^2} \sum_{\substack{(n,2a)=1 \\ \l n = \Box}} \frac{\L_f(n)}{n^{\frac{1}{2} + \a}} \frac{\phi(\l n )}{\l n} \int_{-\infty}^{\infty} \omega_\a \left( \frac{n}{8xX}\right) \Phi(x) dx\\
&= \frac{X}{2} \sum_{\substack{(a,2\l)  =1 \\ a \leq Y}}  \frac{\mu(a)}{a^2} \frac{1}{2\pi i} \int_{(1)} \tilde{\Phi}(s+1) Z_1( \tfrac{1}{2} + \a + s,a,\l)  8^s  X^s  g_\a (s)  \frac{G(s)}{s} ds,
\end{split}
\label{equ:M_1plus-k=0-1}
\end{equation}
where 
\[
Z_1(\tfrac{1}{2} +\gamma ,a,\l)  := \sum_{\substack{(n,2a)=1 \\ \l n = \Box}} \frac{\L_f(n)}{n^{\frac{1}{2}+\gamma}} \frac{\phi(\l n )}{\l n} .
\]

For simplicity we use $E_1(\gamma;p),E_2(\gamma;p),E_3(\gamma;p)$ to denote  the three  Euler factors in \eqref{equ:conj1}, \eqref{equ:conj2}, \eqref{equ:conj3}, respectively. For $\Re(\gamma)>0$, write
\begin{align}
\mathcal{A}(\gamma, a,\l) : =\prod_{p |\l_1}   E_1 (\gamma ; p)   
\prod_{\substack{p \nmid \l_1 \\ p |\l_2}} E_2 (\gamma; p) 
\prod_{(p,2a \l)=1}\left(  E_3 (\gamma; p) + \frac{1}{p^2 -1} \right).
\label{def:A}
\end{align}
\begin{lem}
For $\Re(\gamma)>0$, we have
\begin{align*}
Z_1(\tfrac{1}{2} + \gamma,a,\l) = \frac{1}{\l_1^{\frac{1}{2}+\gamma} \zeta_{2a}(2)} \mathcal{A}(\gamma, a,\l).
\end{align*}
\label{lem:Z}
\end{lem}
\begin{proof}
For each prime, let $b_1, b_2$ be integers such that $p^{b_1} || \l_1$ and $p^{b_2} || \l_2$.
Change the variable $n \rightarrow \l_1 n^2$. We can do this because $\l n = \Box$ implies $\l_1n = \Box$. It gives
\begin{align*}
Z_1(\tfrac{1}{2} + \gamma,a,\l)
 &= \frac{1}{\l_1^{\frac{1}{2} +\gamma}}\sum_{(n,2a)=1} \frac{\L_f (\l_1  n^2)}{n^{1+2\gamma}} \prod_{p | \l_1\l_2 n} \left( 1 - \frac{1}{p}\right)\\
 &=\frac{1}{\l_1^{\frac{1}{2} + \gamma}} \prod_{(p,2a) =1} \sum_{r=0}^{\infty} \frac{\L_f (p^{b_1 + 2r})}{p^{(1+2\gamma)r}} \prod_{q| p^{b_1 + b_2 +  r}}\left( 1 - \frac{1}{q}\right).
\end{align*}
In the following, we consider three cases for the sum over $r$ above. 

If $(p,2a\l) =1$, then $b_1=b_2 =0$. Thus
\begin{equation}
\begin{split}
&\sum_{r=0}^{\infty} \frac{\L_f (p^{ 2r})}{p^{(1+ 2\gamma)r}} \prod_{q| p^{  r}}\left( 1 - \frac{1}{q}\right) \\
&=1 +  \left( 1 - \frac{1}{p}\right)  \sum_{r=1}^{\infty} \frac{\L_f (p^{ 2r})}{p^{(\frac{1}{2}+ \gamma)2r}} \\
&= 1 + \left( 1- \frac{1}{p}\right)\\
&\times \left[ \frac{1}{2} \left( 1- \frac{\L_f (p)}{p^{\frac{1}{2} + \gamma}} + \frac{1}{p^{1+2\gamma}}\right) ^{-1} + \frac{1}{2 } \left( 1+ \frac{\L_f (p)}{p^{\frac{1}{2} + \gamma}}+ \frac{1}{p^{1+2\gamma }}\right) ^{-1} -1\right].
\end{split}
\label{equ:Z1-1}
\end{equation}

If $(p,2a) = 1, p|\l_1$, then $b_1 = 1$ since $\l_1$ is square-free. Hence 
\begin{equation}
\begin{split}
&\sum_{r=0}^{\infty} \frac{\L_f (p^{1 + 2r})}{p^{(1+2\gamma)r }} \left( 1 - \frac{1}{p}\right) \\
&= p^{\frac{1}{2} + \gamma}\left( 1 - \frac{1}{p}\right) \sum_{r=0}^{\infty} \frac{\L_f (p^{1 + 2r})}{p^{(\frac{1}{2}+\gamma ) (1+2r)}}    \\
&=  p^{\frac{1}{2} + \gamma} \left( 1- \frac{1}{p}\right)\\
&\times  \left[ \frac{1}{2} \left( 1- \frac{\L_f (p)}{p^{\frac{1}{2} + \gamma}} + \frac{1}{p^{1+2\gamma}}\right) ^{-1}  - \frac{1}{2 } \left( 1+ \frac{\L_f (p)}{p^{\frac{1}{2} + \gamma}}+ \frac{1}{p^{1+2\gamma}}\right) ^{-1} \right].
\end{split}
\label{equ:Z1-2}
\end{equation}

If $(p,2a) = 1, p\nmid \l_1, p|\l_2$, then $b_1 = 0, b_2 \geq 1$. This gives 
\begin{equation}
\begin{split}
&\sum_{r=0}^{\infty} \frac{\L_f (p^{ 2r})}{p^{(1+2\gamma)r }} \left( 1 - \frac{1}{p}\right) 
\\
&=  
\left( 1- \frac{1}{p}\right)\left[ \frac{1}{2} \left( 1- \frac{\L_f (p)}{p^{\frac{1}{2} + \gamma}} + \frac{1}{p^{1+2\gamma}}\right) ^{-1} + \frac{1}{2 } \left( 1+ \frac{\L_f (p)}{p^{\frac{1}{2} +\gamma}}+ \frac{1}{p^{1+2\gamma}}\right) ^{-1} \right].
\end{split}
\label{equ:Z1-3}
\end{equation}

We then complete the proof  by taking out the factor $1- \frac{1}{p^2}$ from \eqref{equ:Z1-1}, \eqref{equ:Z1-2} and \eqref{equ:Z1-3}, and recalling the definition of $\mathcal{A}(\gamma, a,\l)$ in \eqref{def:A}. 
\end{proof}

It follows from \eqref{equ:M_1plus-k=0-1} and Lemma \ref{lem:Z} that 
\begin{lem}\label{lem:M1}
We have
\begin{align*}
&M_N^+ (\a,\l, k=0) \\
&= \frac{4X}{\pi^2 \l_1^{\frac{1}{2} +\a} } \sum_{\substack{(a,2\l) =1\\ a \leq Y}} \frac{\mu(a)}{a^2} \prod_{p|a} \frac{1}{1-\frac{1}{p^2}}  \frac{1}{2\pi i} \int_{(1)} \tilde{\Phi}(s+1)  
 \mathcal{A}( s + \a, a,\l) \frac{1}{\l_1^s} 8^s  X^s\\
 &\times   g_\a (s)  \frac{G(s)}{s} ds.
\end{align*}
\end{lem}

\section{Upper bound for  $M_N^+ (\a,\l, k\neq 0)$ \label{sec:knot=0}}

We shall prove  an upper bound for $M_N^{+} (\a,\l, k\neq 0)$ in this section.
Recall in \eqref{equ:M-1plus} that 
\begin{equation}
\begin{split}
M_N^+ (\a,\l, k\neq 0) &= \frac{X}{2} \sum_{\substack{(a,2\l) =1\\ a \leq Y}}  \frac{\mu(a)}{a^2} \sum_{(n,2a)=1} \frac{\L_f(n)}{n^{\frac{1}{2} + \a}} \sum_{k \neq 0} (-1)^k  \frac{G_k(\l n )}{\l n} \\
& \times \int_{-\infty}^{\infty} (\cos + \sin ) \left( \frac{2 \pi k x X}{2 n \l a^2}\right) \omega_\a \left( \frac{n}{8xX}\right) \Phi(x) dx.
\end{split}
\label{qu:M_1plus-k!=0-1}
\end{equation}

\begin{lem}
Let $f(x)$ be a smooth function on $\mathbb{R}_{> 0}$. Suppose $f$ decays rapidly as $x \rightarrow \infty$, and $f^{(n)}(x)$ converges as $x \rightarrow 0^+$ for every $n\in \mathbb{Z}_{\geq 0}$.
Then we have 
\begin{equation}
\int_0^{\infty} f(x)\cos (2\pi xy) dx=\frac{1}{2\pi i} \int_{(\frac{1}{2})} \tilde{f}(1-u) \Gamma(u) \cos \left(\frac{\operatorname{sgn}(y)\pi u}{2}\right) (2\pi|y|)^{-u}du.
\label{equ:cos-sin-2}
\end{equation}
In addition, the equation \eqref{equ:cos-sin-2} is also valid when  $\cos$ is replaced by $\sin$.
\label{lem:mellin}
\end{lem}
\begin{proof}
See \cite[Section 3.3]{Sound-Young}.
\end{proof}

By  Lemma \ref{lem:mellin}, the integral in \eqref{qu:M_1plus-k!=0-1} is 
\begin{align*}
&\frac{1}{2\pi i} \int_{(\frac{1}{2})} X^{-u}\Gamma(u) (\cos + \operatorname{sgn} (k) \sin ) \left( \frac{\pi u}{2}\right)  \left( \frac{\l na^2}{\pi |k|} \right)^u \\
& \times \int_0^\infty \Phi(x) x^{-u} \omega_\alpha\left( \frac{n}{8xX}\right) dx du \\
&=\frac{1}{(2\pi i)^2} \int_{(\frac{1}{2})} \int_{(1)} \tilde{\Phi}(1+s-u) X^{-u+s}  \Gamma(u)(\cos + \operatorname{sgn} (k) \sin ) \left( \frac{\pi u}{2}\right) \left(\frac{\l a^2}{\pi |k|} \right)^u \\
& \times {8^s}\frac{1}{n^{s-u}}g_\alpha(s)  \frac{G(s)}{s} ds du.
 \end{align*}
Move the contour of the above integral to $\Re(u) = \frac{1}{2} + \varepsilon, \Re(s) = \frac{1}{2} + 2\varepsilon$, and  change the variable $s'= s - u$. This implies
\begin{align*}
&\frac{1}{(2\pi i)^2} \int_{(\frac{1}{2}+ \varepsilon)} \int_{(\varepsilon)} \tilde{\Phi}(1+s) X^{s}  \Gamma(u)(\cos + \operatorname{sgn} (k) \sin ) \left( \frac{\pi u}{2}\right) \left(\frac{\l a^2}{\pi |k|} \right)^u \\
& \times {8^{s+u}} \frac{1}{n^{s}}g_\alpha(s+u)  \frac{G(s+u)}{s+u} ds du.
 \end{align*}
Together with \eqref{qu:M_1plus-k!=0-1}, it follows that 
\begin{equation}
\begin{split}
&M_N^+ (\a,\l, k\neq 0) \\
&= \frac{X}{2\l } \sum_{\substack{(a,2\l) =1\\ a \leq Y}} \frac{\mu(a)}{a^2}  \sum_{k \neq 0} (-1)^k  \frac{1}{(2\pi i)^2} \int_{(\frac{1}{2}+ \varepsilon)} \int_{(\varepsilon)} \tilde{\Phi}(1+s) X^{s}  \Gamma(u)\\
&\times (\cos + \operatorname{sgn} (k) \sin ) \left( \frac{\pi u}{2}\right) \left(\frac{\l a^2}{\pi |k|} \right)^u\\
& \times\ {8^{s+u}} g_\alpha(s+u) \frac{G(s+u)}{s+u} Z_2(\tfrac{1}{2} + \a + s, a,k,\l)ds du,
\end{split}
\label{equ:M_1plus-k!=0-2}
\end{equation}
where 
\begin{align*}
Z_2(\gamma, a,k,\l) := \sum_{(n,2a)=1} \frac{\L_f(n)}{n^{\gamma}} \frac{G_k(\l n )}{ n}.
\end{align*}
\begin{lem}
Write $4k =k_1 k_2^2$,  where $k_1$ is a fundamental discriminant (possibly $k_1 =1$) and $k_2$ is positive. Then
for $\Re(\gamma) >\frac{1}{2}$, we have
\begin{align}
Z_2(\gamma, a,k,\l) = L(\tfrac{1}{2} + \gamma, f\otimes \chi_{k_1}) Z_3 (\gamma, a,k,\l).
\label{equ:Z-3}
\end{align}
Here 
\[
Z_3 (\gamma, a,k,\l):= \prod_{p} Z_{3,p} (\gamma, a,k,\l),
\] 
where 
\begin{align*}
&Z_{3,p} (\gamma, a,k,\l) :=  1 - \frac{\L_f (p) \chi_{k_1} (p)}{p^{\frac{1}{2} + \gamma}} + \frac{\chi_{k_1} (p)^2}{p^{1+ 2\gamma}}\quad \text{if } p|2a, \quad \text{and} \\
&Z_{3,p} (\gamma, a,k,\l) \\
&:= \left( 1 - \frac{\L_f (p) \chi_{k_1} (p)}{p^{\frac{1}{2} + \gamma}} + \frac{\chi_{k_1} (p)^2}{p^{1+ 2\gamma}}\right) \sum_{r=0}^{\infty} \frac{\L_f (p^r) }{p^{r\gamma}} \frac{G_k (p^{r+ \operatorname{ord}_p(\l)})}{p^r}\quad \text{if } p\nmid 2a.
\end{align*}

Moreover, $Z_3 (\gamma, a,k,\l) $ is analytic in the region $\Re(\gamma) > 0$ and is uniformly bounded  by 
$a^\varepsilon |k| ^ \varepsilon \l^{\frac{1}{2} + \varepsilon} (\l, k_2^2 )^{\frac{1}{2}}$ in the region $\Re(\gamma) > \frac{\varepsilon}{2}$.
\label{lem:Z-3}
\end{lem}
\begin{proof}
The proof is similar to \cite[Lemma 5.3]{Sound00}.  Note that  $G_k(n)$ is multiplicative. Hence
\[
Z_2(\gamma, a,k,\l) = \prod_{(p,2a)=1} \sum_{r=0} \frac{\L_f (p^r)}{p^{r\gamma}} \frac{G_k(p^{r+ \operatorname{ord}_p(\l) })}{p^r}.
\]
Then the identity \eqref{equ:Z-3}  follows directly from a comparison of  both sides.

When $p \nmid 2ak\l$, by the definition of $Z_{3,p} (\gamma, a,k,\l)$ and Lemma \ref{lem:preciseG}, we know
\begin{equation}
\begin{split}
Z_{3,p} (\gamma, a,k,\l) & = \left( 1 - \frac{\L_f (p) \chi_{k_1} (p)}{p^{\frac{1}{2} + \gamma}} + \frac{\chi_{k_1} (p)^2}{p^{1+ 2\gamma}}\right) \left( 1 + \frac{\L_f(p) \chi_{k_1}(p)}{p^{\frac{1}{2} + \gamma}}\right) \\
&=1 + \frac{\chi_{k_1}(p)^2}{p^{1+ 2\gamma}} - \frac{\L_f(p)^2 \chi_{k_1}(p)^2}{p^{1+2\gamma}} + \frac{\L_f(p)\chi_{k_1}(p)^3}{p^{\frac{3}{2} + 3\gamma}}.
\end{split}
\label{equ:Z-3-1}
\end{equation}
Hence $Z_3 (\gamma, a,k,\l) $ is analytic in the region $\Re(\gamma)>0$.

It remains to prove the upper bound of $Z_3 (\gamma, a,k,\l)$. For 
$p \nmid 2a k\l$, by \eqref{equ:Z-3-1} and the fact $|\L_f(n)| \leq \tau(n)$,  we get
\begin{align}
\prod_{(p,2ak\l)=1}Z_{3,p} (\gamma, a,k,\l) \ll  1.
\label{equ:Z4-1}
\end{align}
For $p|2a$,  we have
\begin{align}
\prod_{p|2a} Z_{3,p}(\gamma, a,k,\l)  \ll   a^{\varepsilon}.
\label{equ:Z4-2}
\end{align}
For $p \nmid 2a, p|k\l$, we let $p^{b_1} || k, p^{b_2} ||\l$. 
We  can assume $b_2 \leq b_1+1$ since  $G_k(p^{r+b_2}) = 0$ otherwise (by Lemma \ref{lem:preciseG}).  We claim  $Z_{3,p} (\gamma, a,k,\l) \ll (1+ b_1 + b_2)^2 p^{\min(b_2, \floor{\frac{b_1}{2}} + \frac{b_2}{2})}$. In fact,  the trivial bound $G_k(p^n) \leq  p^n$ gives $Z_{3,p} (\gamma, a,\allowbreak k,\l)  \ll  (1+ b_1 + b_2)^2 p^{b_2}$, which proves the case $b_2 \leq \floor{\frac{b_1}{2}} + \frac{b_2}{2}$. The remaining cases include: $b_1$ even and $b_2 =b_1 + 1$, or $b_1$ odd and $b_2=b_1$, or $b_1$ odd  and $b_2 = b_1 + 1$. For $b_1$ even and $b_2 =b_1 + 1$, by Lemma \ref{lem:preciseG}, we know $Z_{3,p} (\gamma, a,k,\l) \ll  p^{b_1 } \sqrt{p}= p^{\floor{\frac{b_1}{2}} + \frac{b_2}{2}}$. The other two cases can be done similarly.
This combined with \eqref{equ:Z4-1} and \eqref{equ:Z4-2} gives the upper bound for $Z_3 (\gamma, a,k,\l)$.
\end{proof}
By   \eqref{equ:M_1plus-k!=0-2}   and Lemma \ref{lem:Z-3}, we have
\begin{align*}
&M_N^+ (\a,\l, k\neq 0) \nonumber\\
&= \frac{X}{2\l } \sum_{\substack{(a,2\l) =1\\ a \leq Y}} \frac{\mu(a)}{a^2}  \sum_{k \neq 0} (-1)^k  \frac{1}{(2\pi i)^2} \int_{(\frac{1}{2}+ \varepsilon)} \int_{(\varepsilon)}\tilde{\Phi}(1+s) X^{s}  \Gamma(u)\\
& \times (\cos + \operatorname{sgn} (k) \sin ) \left( \frac{\pi u}{2}\right) \left(\frac{\l a^2}{\pi |k|} \right)^u \nonumber\\
& \times  {8^{s+u}}  g_\alpha(s+u) \frac{ G(s+u)}{s+u} L(1 + \alpha +s , f\otimes \chi_{k_1}) Z_3 (\tfrac{1}{2} + \alpha +s, a,k,\l)ds du.
\end{align*}
Move the lines of the integral to $\Re(s) = -\frac{1}{2} - \alpha+ \varepsilon$, $\Re(u) = 1+ \varepsilon$ without encountering any poles. Together with Lemma \ref{HB-bd} and Lemma \ref{lem:Z-3}, it follows that
\begin{lem} We have
\begin{align*}
M_N^{+} (\a,\l, k\neq 0) \ll \l^{\frac{1}{2}+\varepsilon }    X^{\frac{1}{2}+ \varepsilon}Y.
\end{align*}
\label{lem:non-diagnol}
\end{lem}

\section{Evaluation of $M_R^{+} (\a,\l)$\label{sec:Error}}
In this section we shall simplify  $M_{R,1}^{+} (\a,\l)$, and derive upper bounds for $M_{R,2}^{+} (\a,\l), M_{R,3}^{+} (\a,\l)$ by proving the follow lemma.
\begin{lem}\label{lem:M2} 
We have
\begin{align}
M_{R,1}^+ (\a,\l)
&=\frac{4X }{\pi^2\l_1^{\frac{1}{2} + \alpha}}  \sum_{\substack{a>Y \\ (a,2\l)=1}} \frac{\mu (a)}{a^2}  \prod_{p|a} \frac{1}{1-\frac{1}{p^2}} \frac{1}{2 \pi i}  \int_{(1)}     
\tilde{\Phi}(1+s)
\mathcal{A}(\a + s, a,\l) \nonumber \\
& \times 
 \frac{1}{\l_1^s}X^s 8^sg_\a (s)\frac{G(s)}{s} ds,  \label{equ:M21-1}\\
M_{R,2}^{+} (\a,\l) &\ll     \l^{\varepsilon} X^{\frac{1}{2} + \varepsilon} Y, \label{equ:M21-2}     \\
M_{R,3}^{+} (\a,\l) &\ll\l^{\frac{1}{2} + \varepsilon} \frac{X^{h+\varepsilon}}{Y^{2h-1}}.  \label{equ:M21-3}
\end{align}
\end{lem}
We give a proof for the above lemma in the rest of the section. Recall $M_{R,1}^{+} (\a,\l)$ in \eqref{equ:M21}.  By interchanging summations and integrals, we know
\begin{equation}
\begin{split}
&M_{R,1}^+ (\a,\l)\\
&=\frac{4X }{\pi^2\l_1^{\frac{1}{2} + \alpha}}  \sum_{\substack{a>Y \\ (a,2\l)=1}} \mu (a)\frac{1}{2 \pi i}  \int_{(1)}   \tilde{\Phi}(1+s)     \sum_{(r_1,2\l)=1 }  \frac{\mu(r_1)  \L_f (r_1) }{r_1^{1+ 2\alpha+2s }}   
\sum_{(r_2,2\l)=1 }   \frac{\mu( r_1 r_2)^2}{r_2^{1+ 2\a +2s}} \\
& \times  
L(1+2\alpha+2s, \operatorname{sym}^2 f)  Z (\tfrac{1}{2} + \alpha+s,\l r_1 r_2^2)
  \sum_{\substack{(c,2\l)=1\\a,r_1,r_2|c }}  \frac{1}{c^2}
 \frac{1}{\l_1^s}X^s 8^sg_\a (s)\frac{G(s)}{s} ds. 
 \end{split}
 \label{equ:MR1}
\end{equation}
\begin{lem}\label{lem:complic}
For $\Re(\gamma)>0$,
\begin{equation}
\begin{split}
&\sum_{(r_1,2\l)=1 }  \frac{\mu(r_1)  \L_f (r_1) }{r_1^{1+ 2\gamma }}   
\sum_{(r_2,2\l)=1 }   \frac{\mu( r_1 r_2)^2}{r_2^{1+ 2\gamma}}\\
 &\times  \sum_{\substack{(c,2\l)=1\\a,r_1,r_2|c }}  \frac{1}{c^2}   L(1+2\gamma, \operatorname{sym}^2 f)  Z (\tfrac{1}{2} + \gamma,\l r_1r_2^2)  \\
  &=
 \frac{1}{a^2} \prod_{p|a} \frac{1}{1-\frac{1}{p^2}} \mathcal{A}(\gamma, a, \l).
 \end{split}
  \label{equ:complic}
\end{equation}
\end{lem}
\begin{proof}
 The left-hand side of \eqref{equ:complic} is 
\begin{equation}
\begin{split}
&\sum_{(r_1,2\l)=1 }  \frac{\mu(r_1)  \L_f (r_1) }{r_1^{1+ 2\gamma }}   
\sum_{(r_2,2\l)=1 }   \frac{\mu( r_1 r_2)^2}{r_2^{1+ 2\gamma}}
  \sum_{\substack{(c,2\l)=1\\a,r_1,r_2|c }}  \frac{1}{c^2}
  \prod_{p| \l_1r_1} E_1(\gamma;p)  \\
  &\times \prod_{\substack{p \nmid \l_1r_1\\p|\l_2 r_2}} E_2(\gamma;p) \prod_{(p,2\l r_1r_2)=1} E_3(\gamma; p).
  \end{split}
  \label{equ:complic1}
\end{equation}
Note 
\begin{align}
 \sum_{\substack{(c,2\l)=1\\a,r_1,r_2|c }}  \frac{1}{c^2} = \sum_{\substack{[r_1, r_2, a] |c \\ (c,2 \l)=1}} \frac{1}{c^2}
 = \frac{1}{[r_1, r_2, a]^2} \sum_{(c,2\l)=1}\frac{1}{c^2} =\frac{1}{a^2} \frac{([r_1,r_2],a)^2}{[r_1,r_2]^2} \zeta_{2\l}(2).
   \label{equ:complic2}
\end{align}
We also see that
\begin{align}
\prod_{\substack{p \nmid \l_1r_1\\p|\l_2 r_2}} E_2(\gamma;p)  = \prod_{p|\l_2r_2} E_2(\gamma;p)  \prod_{p|(r_1,r_2)} E_2(\gamma;p) ^{-1} \prod_{p|(\l_1,\l_2)} E_2(\gamma;p)^{-1}.
  \label{equ:complic3}
\end{align}
Inserting \eqref{equ:complic2} and \eqref{equ:complic3} into \eqref{equ:complic1},   the expression  \eqref{equ:complic1}  now is
\begin{align}
\frac{\zeta_{2\l}(2)}{a^2}\prod_{p| \l_1} E_1(\gamma;p)\prod_{\substack{p|\l_2 \\ p\nmid \l_1}} E_2(\gamma;p) \prod_{(p,2\l)=1} E_3(\gamma; p) \sum_{(r_1,2\l)=1 }\sum_{(r_2,2\l)=1 }  H(r_1,r_2),
\label{equ:complic4}
\end{align}
where 
\begin{align*}
&H(r_1,r_2)\\
& :=  \frac{\mu(r_1)  \L_f (r_1) }{r_1^{1+ 2\gamma }}   
  \frac{\mu( r_1 r_2)^2}{r_2^{1+ 2\gamma}}
 \frac{([r_1,r_2],a)^2}{[r_1,r_2]^2}  
 \prod_{p| r_1} E_1(\gamma;p)   \prod_{p|r_2} E_2(\gamma;p) \\
 &\times  \prod_{p|(r_1,r_2)} E_2(\gamma;p) ^{-1}
  \prod_{p|r_1r_2} E_3(\gamma; p)^{-1}.
\end{align*}
Clearly $H(r_1,r_2)$ is joint multiplicative. Then
\begin{align*}
&\sum_{(r_1,2\l)=1 }\sum_{(r_2,2\l)=1 }  H(r_1,r_2) \\
&= \prod_{(p,2\l)=1} \left( 1 -  \frac{\L_f(p)}{p^{1+2\gamma}} \frac{(p,a)^2}{p^2}E_1(\gamma;p) E_3(\gamma;p) ^{-1} \right.\\
&+ \left. \frac{1}{p^{1+2\gamma}} \frac{(p,a)^2}{p^2} E_2(\gamma;p)E_3(\gamma;p) ^{-1} \right).
\end{align*}
It follows that
\begin{equation}
\begin{split}
&\prod_{(p,2\l)=1} E_3(\gamma;p)  \cdot \sum_{(r_1,2\l)=1 }\sum_{(r_2,2\l)=1 }  H(r_1,r_2)  \\
&= \prod_{(p,2a\l)=1} \left( E_3(\gamma;p)-  \frac{\L_f(p)}{p^{3+2\gamma}} E_1(\gamma;p)+ \frac{1}{p^{3+2\gamma}} E_2(\gamma;p) \right)\prod_{p|a}1 \\
&= \prod_{(p,2a\l)=1}\Bigg\{1 + \frac{p}{p+1} \frac{1}{2} \left( 1- \frac{\L_f (p)}{p^{\frac{1}{2} + \gamma} } + \frac{1}{p^{1+2\gamma}} \right) ^{-1} 
\left( 1 - \frac{\L_f(p)}{p^{\frac{1}{2}+ \gamma}} \frac{1}{p^2} + \frac{1}{p^{1+ 2\gamma}} \frac{1}{p^2}\right)
 \\
&+  \frac{p}{p+1} \frac{1}{2} \left( 1+ \frac{\L_f (p)}{p^{\frac{1}{2} + \gamma} } + \frac{1}{p^{1+2\gamma}} \right) ^{-1} 
\left( 1 + \frac{\L_f(p)}{p^{\frac{1}{2}+ \gamma}} \frac{1}{p^2} + \frac{1}{p^{1+ 2\gamma}} \frac{1}{p^2}\right) - \frac{p}{p+1} \Bigg\}.
\end{split}
\label{equ:complic5--}
\end{equation}
The last identity above is obtained by grouping terms involving $ \frac{p}{p+1} \frac{1}{2} ( 1\mp \frac{\L_f (p)}{p^{\frac{1}{2} + \gamma} } + \frac{1}{p^{1+2\gamma}} ) ^{-1} $.
Note that 
\begin{align*}
1 - \frac{\L_f(p)}{p^{\frac{1}{2}+ \gamma}} \frac{1}{p^2} + \frac{1}{p^{1+ 2\gamma}} \frac{1}{p^2} &= 1-\frac{1}{p^2} + \frac{1}{p^2}\left(1- \frac{\L_f (p)}{p^{\frac{1}{2} + \gamma} } + \frac{1}{p^{1+2\gamma}}  \right),\\
1 + \frac{\L_f(p)}{p^{\frac{1}{2}+ \gamma}} \frac{1}{p^2} + \frac{1}{p^{1+ 2\gamma}} \frac{1}{p^2} &= 1-\frac{1}{p^2} + \frac{1}{p^2}\left(1+ \frac{\L_f (p)}{p^{\frac{1}{2} + \gamma} } + \frac{1}{p^{1+2\gamma}}  \right).
\end{align*}
Thus, \eqref{equ:complic5--} can be simplified to
\begin{equation}
\begin{split}
&\prod_{(p,2a\l)=1} \left\{ 1 + \frac{p}{p+1} \frac{1}{2} \left( 1- \frac{\L_f (p)}{p^{\frac{1}{2} + \gamma} } + \frac{1}{p^{1+2\gamma}} \right) ^{-1} 
\left( 1 - \frac{1}{p^2}\right) + \frac{p}{p+1} \frac{1}{2} \frac{1}{p^2}\right.
 \\
&\left. +  \frac{p}{p+1} \frac{1}{2} \left( 1+ \frac{\L_f (p)}{p^{\frac{1}{2} + \gamma} } + \frac{1}{p^{1+2\gamma}} \right) ^{-1} 
\left( 1 - \frac{1}{p^2}\right) + \frac{p}{p+1} \frac{1}{2} \frac{1}{p^2}
 - \frac{p}{p+1} \right\}
 \\
&= \prod_{(p,2a\l)=1} \left( 1 - \frac{1}{p^2}\right) 
\left(
E_3 (\gamma; p)
 +  \frac{1}{p^2 -1}
 \right).  
 \end{split}
 \label{equ:complic5}
\end{equation}
Substituting \eqref{equ:complic5} in  \eqref{equ:complic4} completes the proof of Lemma \ref{lem:complic}.
 \end{proof}
We then can complete the proof for  \eqref{equ:M21-1} by using  \eqref{equ:MR1} and  Lemma \ref{lem:complic}.

Next recall $M_{R,2}^+ (\a,\l)$ in \eqref{equ:M22}, which   is of the form 
\[
M_{R,2}^+ (\a,\l) = \frac{ i^\kappa}{\l_1^{\frac{1}{2} - \alpha }} \sum_{(c,2\l)=1} \sum_{\substack{a>Y \\ a|c}} \mu (a)   T(c,\a,\l).
\]
We extend the sum over $a>Y$  to that over all positive integers. Then
\begin{align}
M_{R,2}^+ (\a,\l) = \frac{ i^\kappa}{\l_1^{\frac{1}{2} - \alpha }} \sum_{(c,2\l)=1} \sum_{a|c} \mu (a)   T(c,\a,\l)-  \frac{i^\kappa}{\l_1^{\frac{1}{2} - \alpha }} \sum_{(c,2\l)=1} \sum_{\substack{a\leq Y \\ a|c}} \mu (a)   T(c,\a,\l).
\label{equ:M22-1}
\end{align}
We know  $\sum_{a|c} \mu (a) = 1$ when $c=1$, and is zero  otherwise. Thus
\begin{align*}
&\frac{ i^\kappa}{\l_1^{\frac{1}{2} - \alpha }} \sum_{(c,2\l)=1} \sum_{a|c} \mu (a)   T(c,\a,\l)  \\
&= \frac{i^\kappa}{\l_1^{\frac{1}{2} - \alpha }}  
 \frac{1}{2 \pi i}  \int_{(\frac{1}{\log X})}       \frac{4X ^{1-2\a-s}\gamma_{\a +s}\tilde{\Phi}(1-2\a-s)}{\pi^2} L(1-2\alpha-2s, \operatorname{sym}^2 f)  \nonumber \\
& \times Z (\tfrac{1}{2} - \alpha-s,\l )
\l_1^s  8^sg_\a (s)\frac{G(s)}{s} ds. 
\end{align*}
Move the line of the above  integral  from $\Re(s) = \frac{1}{\log X}$ to $\Re(s) =  \frac{1}{2} - \varepsilon$. We encounter no poles due to Lemma \ref{lem:lastlem} and Remark \ref{rem:G}. It follows that 
\begin{align}
\frac{ i^\kappa}{\l_1^{\frac{1}{2} - \alpha }} \sum_{(c,2\l)=1} \sum_{a|c} \mu (a)   T(c,\a,\l)  \ll \l^\varepsilon    X^{\frac{1}{2} + \varepsilon}.
\label{equ:M22-2}
\end{align}
For the second term of  \eqref{equ:M22-1},  we move the contour of the integral in $T(c,\a,\l)$ to $\Re(s) = \frac{1}{10}$ without encountering any poles.  We have 
\begin{align*}
&\frac{ i^\kappa}{\l_1^{\frac{1}{2} - \alpha }} \sum_{(c,2\l)=1} \sum_{\substack{a\leq Y \\ a|c}} \mu (a)   T(c,\a,\l)\\
&=\frac{i^\kappa}{\l_1^{\frac{1}{2} - \alpha }} \sum_{\substack{a \leq Y \\(a,2\l)=1}}  \mu (a)  \sum_{(r_1,2\l)=1} \frac{\mu(r_1) \L_f (r_1)}{r_1} \sum_{(r_2,2\l)=1} \frac{\mu(r_1 r_2)^2}{r_2^{1+2\a }}\\
&\times \frac{1}{2\pi i}\int_{(\frac{1}{10})}    \frac{4X ^{1-2\a-s}\gamma_{\a +s}\tilde{\Phi}(1-2\a-s)}{\pi^2} \sum_{\substack{(c,2\l)=1 \\ a,r_1,r_2|c}} \frac{1}{c^{2-4\a -4s}} \\
& \times  L(1-2\alpha-2s, \operatorname{sym}^2 f)   Z (\tfrac{1}{2} - \alpha-s,\l r_1 r_2^2 ) \frac{\l_1^s8^s}{r_2^{2s}} g_\a (s)\frac{G(s)}{s}  ds.
\end{align*}
Treat $\sum_{\substack{(c,2\l)=1 \\ a,r_1,r_2|c}} \frac{1}{c^{2-4\a -4s}}$ as in \eqref{equ:complic2}. The above is 
\begin{align*}
&\frac{ i^\kappa}{\l_1^{\frac{1}{2} - \alpha }} \sum_{\substack{a \leq Y \\ (a,2\l)=1}}  \mu (a)  \sum_{(r_1,2\l)=1} \frac{\mu(r_1) \L_f (r_1)}{r_1} \sum_{(r_2,2\l)=1} \frac{\mu(r_1 r_2)^2}{r_2^{1+2\a }} \\
&\times \frac{1}{2\pi i}\int_{(\frac{1}{10})}    \frac{4X ^{1-2\a-s}\gamma_{\a +s}\tilde{\Phi}(1-2\a-s)}{\pi^2}  \\
& \times \frac{1}{a^{2-4\a -4s}} \frac{([r_1,r_2],a)^{2-4\a -4s}}{[r_1,r_2]^{2-4\a -4s}}   \zeta_{2\l}( 2-4\a -4s)  L(1-2\alpha-2s, \operatorname{sym}^2 f)  \\
& \times  Z (\tfrac{1}{2} - \alpha-s,\l r_1 r_2^2 ) \frac{\l_1^s}{r_2^{2s}} 8^sg_\a (s)\frac{G(s)}{s}  ds.
\end{align*}
Move the contour of the integral above to $\Re(s)= \frac{1}{2} - \varepsilon$ without encountering any poles by Lemma \ref{lem:lastlem} and Remark \ref{rem:G}. In particular, the pole of $\zeta(2-4\a -4s)$ is canceled by the factor $1-4\a -4s$ in $G(s)$. By the fact 
\[
\left| \frac{([r_1,r_2],a)^{2-4\a -4s}}{a^{2-4\a -4s}} \frac{1}{[r_1,r_2]^{2-4\a -4s}} \right| \leq \frac{1}{r_1^\varepsilon},
\]
we obtain 
\begin{align}
\frac{ i^\kappa}{\l_1^{\frac{1}{2} - \alpha }} \sum_{(c,2\l)=1} \sum_{\substack{a\leq Y \\ a|c}} \mu (a)   T(c,\a,\l) \ll \l^{\varepsilon} X^{\frac{1}{2} + \varepsilon} Y.
\label{equ:M22-3}
\end{align}
Combining \eqref{equ:M22-1}, \eqref{equ:M22-2} and \eqref{equ:M22-3} gives \eqref{equ:M21-2}.

Finally, recall $M_{R,3}^+ (\a,\l)$ in \eqref{equ:M23}.  Note $h \geq \frac{1}{2}$. Then
\begin{align*}
M_{R,3}^+ (\a,\l) 
\ll \l^{\frac{1}{2} +\frac{ \varepsilon}{10}}\sum_{(c,2\l)=1} \left( \frac{X}{c^2} \right)^{h+\varepsilon} \sum_{\substack{a>Y \\ a|c}}   \sum_{r_1 |c}  
\sum_{r_2 |c}  (r_1 r_2^2 )^{\frac{\varepsilon}{10}}  \ll  \l^{\frac{1}{2} + \varepsilon} \frac{X^{h+\varepsilon}}{Y^{2h-1}},
\end{align*}
which gives \eqref{equ:M21-3}.

\section{Proof of  Theorem \ref{thm-assump}   \label{sec:comp} }
By Lemmas \ref{lem:M1} and \ref{lem:M2},
\begin{equation}
\begin{split}
&M_{R,1}^+ (\a,\l)  + M_N^+ (\a,\l, k=0) \\
&= \frac{4X}{\pi^2 \l_1^{\frac{1}{2} +\a} } \sum_{(a,2\l) =1} \frac{\mu(a)}{a^2} \prod_{p|a} \frac{1}{1-\frac{1}{p^2}}  \frac{1}{2\pi i} \int_{(1)} \tilde{\Phi}(s+1)  
 \mathcal{A}( s + \a, a,\l) \frac{1}{\l_1^s} \\
 &\times 8^s  X^s  g_\a (s)  \frac{G(s)}{s} ds,
 \end{split}
  \label{equ:final1}
\end{equation}
where $\mathcal{A}( s + \a, a,\l)$ is defined in \eqref{def:A}. It can be deduced that 
\begin{align*}
&\sum_{(a,2\l) =1} \frac{\mu(a)}{a^2} \prod_{p|a}  \frac{1}{1-\frac{1}{p^2}}\mathcal{A}( s + \a, a,\l) \\
&= \prod_{p |\l_1}   E_1 (s + \a ; p)   
\prod_{\substack{p \nmid \l_1 \\ p |\l_2}} E_2 (s + \a; p) 
\sum_{(a,2\l) =1} \frac{\mu(a)}{a^2} \prod_{p|a}  \frac{1}{1-\frac{1}{p^2}}\\
&\times \prod_{(p,2a \l)=1}\left(  E_3 (s + \a; p) + \frac{1}{p^2 -1} \right)\\
&= \prod_{p |\l_1}   E_1 (s + \a ; p)   
\prod_{\substack{p \nmid \l_1 \\ p |\l_2}} E_2 (s + \a; p) 
\prod_{(p,2 \l)=1}\left(  E_3 (s + \a; p) + \frac{1}{p^2 -1} \right) \\
& \times \sum_{(a,2\l) =1} \frac{\mu(a)}{a^2} \prod_{p|a}  \frac{(E_3 (s + \a; p) + \frac{1}{p^2 -1})^{-1}}{(1-\frac{1}{p^2})}  
\\
&= \prod_{p|\l_1}E_1  (\a +s;p )\prod_{\substack{p \nmid \l_1 \\ p|\l_2}}E_2(\a +s;p ) \prod_{(p,2\l)=1} E_3 (\a +s;p ) \\
&=L(1+2\alpha + 2s, \operatorname{sym}^2 f)  Z(\tfrac{1}{2} + \alpha+s, \l).
\end{align*}
The second last equation is due to the multiplicity of the function $\frac{\mu(a)}{a^2} \prod_{p|a}$.
This combined with \eqref{equ:final1} gives 
\begin{equation}
\begin{split}
&M_{R,1}^+ (\a,\l)  + M_N^+ (\a,\l, k=0) \\
&= \frac{4X}{\pi^2 \l_1^{\frac{1}{2} +\a} } \frac{1}{2\pi i} \int_{(1)} \tilde{\Phi}(s+1)  
L(1+2\alpha + 2s, \operatorname{sym}^2 f)  Z (\tfrac{1}{2} + \alpha+s, \l)\frac{1}{\l_1^s} \\
&\times 8^s  X^s  g_\a (s)  \frac{G(s)}{s} ds.
\end{split}
\label{equ:above-final2}
\end{equation}
Move the integration to the line $\Re(s) = -\frac{1}{2} + \varepsilon$ with encountering one simple pole at $s=0$ by Lemma \ref{lem:lastlem} and Remark \ref{rem:G}. This gives
\begin{equation}
\begin{split}
&M_{R,1}^+ (\a,\l)  + M_N^+ (\a,\l, k=0) \\
&= \frac{4X}{\pi^2 \l_1^{\frac{1}{2} +\a} } \tilde{\Phi}(1)  L(1+2\alpha , \operatorname{sym}^2 f)  Z (\tfrac{1}{2} + \alpha, \l) + O(X^{\frac{1}{2} + \varepsilon} \l^\varepsilon).
\end{split}
\label{equ:final2}
\end{equation}
By Remark \ref{rem:Phi}, we know
\begin{equation}
\begin{split}
&M_{R,1}^- (\a,\l)  + M_N^- (\a,\l, k=0) \\
&=  i^\kappa\frac{4\gamma_\a X^{1-2\a}}{\pi^2 \l_1^{\frac{1}{2} -\a} } \tilde{\Phi}(1-2\a)  L(1-2\alpha , \operatorname{sym}^2 f)  Z (\tfrac{1}{2} -\alpha, \l) + O(X^{\frac{1}{2} + \varepsilon} \l^\varepsilon). 
\end{split}
\label{equ:final3}
\end{equation}
Similarly we can derive  same upper bounds for $M_N^{-} (\a,\l, k\neq 0)$, $M_{R,2}^{-}  (\a,\l)$ and $M_{R,3}^{-} (\a,\l)$ as  those for $M_N^{+} (\a,\l, k\neq 0)$, $M_{R,2}^{+}  (\a,\l)$ and $M_{R,3}^{+ } (\a,\l)$ in Lemmas \ref{lem:non-diagnol} and \ref{lem:M2}. Therefore it follows from \eqref{relation1}, \eqref{relation2},  \eqref{equ:final2}, \eqref{equ:final3}, and Lemmas \ref{lem:non-diagnol} and \ref{lem:M2} that 
\begin{align*}
M(\a,\l) 
&=\frac{4X}{\pi^2 \l_1^{\frac{1}{2} +\a} } \tilde{\Phi}(1)  L(1+2\alpha , \operatorname{sym}^2 f)  Z (\tfrac{1}{2} + \alpha, \l)  \\
& + i^\kappa \frac{4\gamma_\a X^{1-2\a}}{\pi^2 \l_1^{\frac{1}{2} -\a} } \tilde{\Phi}(1 - 2\a)  L(1-2\alpha , \operatorname{sym}^2 f)  Z (\tfrac{1}{2} -\alpha, \l)  \\
& + O (\l^{\frac{1}{2}+\varepsilon}X^{\frac{1}{2} + \varepsilon} Y) + O \left (\l^{\frac{1}{2} + \varepsilon} \frac{X^{h+\varepsilon}}{Y^{2h-1}} \right).
\end{align*}
Taking $Y= X^{\frac{2h-1}{4h}}$ completes the proof of Theorem \ref{thm-assump}.

\subsection*{Acknowledgements}
This paper is part of my Ph.D. thesis at the University of Lethbridge. I would like to thank  my supervisors Habiba Kadiri and Nathan Ng for their constant encouragement  and many  discussions on this article. I am  grateful to Amir Akbary, Andrew Fiori and Caroline Turnage-Butterbaugh for helpful suggestions, and to the referee for their valuable comments. This work was supported by the NSERC Discovery grants RGPIN-2020-06731 of Habiba Kadiri and RGPIN-2020-06032 of Nathan Ng.


\end{document}